\documentclass[11pt]{article}

\usepackage{amsmath}
\usepackage{amsfonts}
\usepackage{amssymb}
\usepackage{mathtools}
\usepackage{stackrel}
\usepackage{leftidx}
\usepackage[amsmath,amsthm,thmmarks]{ntheorem}

\usepackage{setspace}
\AtBeginEnvironment{thebibliography}{\small\setstretch{0.9}}

\usepackage{authblk}

\usepackage[charter]{mathdesign}
\usepackage{euscript}
\usepackage{dsfont}
\usepackage{bbold}

\DeclareSymbolFont{usualmathcal}{OMS}{cmsy}{m}{n}
\DeclareSymbolFontAlphabet{\mathcal}{usualmathcal}

\usepackage[shortlabels]{enumitem}

\usepackage{hyperref}
\usepackage{cleveref}

\usepackage{xcolor}
\usepackage{tcolorbox}
\usepackage{graphicx}
\usepackage{subfigure}

\usepackage[all,2cell,arrow,matrix]{xy} \UseAllTwocells \SilentMatrices
\usepackage{tikz-cd}
\usepackage{tikz}
\usetikzlibrary{arrows,arrows.meta,decorations.markings,decorations.pathreplacing,calc}

\usepackage{verbatim}
\usepackage{comment}

\usepackage{epstopdf} % for pdflatex

\usepackage[nottoc]{tocbibind} % References in ToC

\usepackage{appendix}

\usepackage{geometry}
\newcommand\void[1]       {}

\tikzset{->-/.style={decoration={markings,mark=at position #1 with {\arrow{stealth}}},postaction={decorate}},->-/.default=0.55}
\colorlet{e_ext}{red}
\colorlet{m_ext}{blue!30}
\tikzset{e_str/.style={very thick,red!80}}
\tikzset{m_str/.style={very thick,blue!80}}
\tikzset{m_dual_str/.style={thick,dashed,blue}}
\tikzset{link_label/.style={scale=0.8,black}}

\theoremstyle{definition}

\newtheorem{thm}{Theorem}[section]
\newtheorem{prop}[thm]{Proposition}

\newtheorem{cor}[thm]{Corollary}
\newtheorem{lem}[thm]{Lemma}

\newtheorem{prob}[thm]{Problem}

{
\newtheorem{defn}[thm]{Definition}
}
{
\newtheorem{expl}[thm]{Example}
}
{
\newtheorem{rem}[thm]{Remark}
}
\qedsymbol{\mbox{$\square$}}

\numberwithin{equation}{section}
\newcommand\be            {\begin{equation}}
\newcommand\ee            {\end{equation}}
\newcommand\bea           {\begin{eqnarray}}
\newcommand\eea         {\end{eqnarray}}
\newcommand\bnu          {\begin{enumerate}}
\newcommand\enu          {\end{enumerate}}
\newcommand\bit          {\begin{itemize}}
\newcommand\eit          {\end{itemize}}

\newcommand{\pf}{\begin{proof}}
\newcommand{\epf}{\qed\end{proof}}

\makeatletter
\providecommand{\leftsquigarrow}{%
  \mathrel{\mathpalette\reflect@squig\relax}%
}
\newcommand{\reflect@squig}[2]{%
  \reflectbox{$\m@th#1\rightsquigarrow$}%
}
\makeatother

\DeclareMathAlphabet{\mathcal}{OMS}{cmsy}{m}{n}	% use the original mathcal font
\DeclareMathAlphabet{\mathsf}{OT1}{cmss}{m}{n}	% use the original mathsf font

\newcommand\Cb			{\mathbb{C}}
\newcommand\Nb			{\mathbb{N}}

\newcommand\Zb			{\mathbb{Z}}

\newcommand\CA			{\EuScript{A}}
\newcommand\CB			{\EuScript{B}}
\newcommand\CC			{\EuScript{C}}

\newcommand\CE			{\EuScript{E}}
\newcommand\CF			{\EuScript{F}}

\newcommand\CL			{\EuScript{L}}
\newcommand\CM			{\EuScript{M}}

\newcommand\CP			{\EuScript{P}}
\newcommand\CQ			{\EuScript{Q}}

\newcommand\CX			{\EuScript{X}}
\newcommand\CY			{\EuScript{Y}}

\newcommand{\FZ}			{\text{\usefont{U}{euf}{m}{n}Z}}

\newcommand\BZ			{\mathbf{Z}}

\DeclareMathOperator{\End}{End}

\DeclareMathOperator{\id}{id}

\DeclareMathOperator{\tr}{tr}

\DeclareMathOperator{\ev}{ev}
\DeclareMathOperator{\coev}{coev}

\DeclareMathOperator{\fun}{Fun}

\DeclareMathOperator{\RMod}{RMod}

\newcommand{\rev}			{\mathrm{rev}}

\newcommand\vect			{\mathrm{Vec}}

\newcommand\rep			{\mathrm{Rep}}

\newcommand\Irr			{\mathrm{Irr}}

\newcommand{\Witt}{\mathds{W}}

\newcommand{\Fus}{\mathbf{Fus}}
\newcommand{\SFC}{\mathbf{SFC}}
\newcommand{\BFC}{\mathbf{BFC}}
\newcommand{\NBFC}{\mathbf{NBFC}}

\newcommand{\sfcFus}{{}^{\mathrm{sfc}}\mathbf{Fus}}
\newcommand{\bfcFus}{{}^{\mathrm{bfc}}\mathbf{Fus}}
\newcommand{\nbfcFus}{{}^{\mathrm{nbfc}}\mathbf{Fus}}
\newcommand{\cfrFus}{{}^{\mathrm{cfr}}\mathbf{Fus}}

\newcommand{\fus}{\mathsf{Fus}}
\newcommand{\sfcfus}{{}^{\mathrm{sfc}}\mathsf{Fus}}
\newcommand{\bfcfus}{{}^{\mathrm{bfc}}\mathsf{Fus}}
\newcommand{\nbfcfus}{{}^{\mathrm{nbfc}}\mathsf{Fus}}
\newcommand{\cfrfus}{{}^{\mathrm{cfr}}\mathsf{Fus}}

\newcommand{\sfc}{\mathsf{SFC}}
\newcommand{\bfc}{\mathsf{BFC}}
\newcommand{\nbfc}{\mathsf{NBFC}}

\newcommand\arXiv[1]{\href{http://arxiv.org/abs/#1}{\nolinkurl{arXiv:#1}}}

\newcommand{\gt}{\operatorname{GT}}
\newcommand{\mrt}{\sim}
\newcommand{\weq}{\approx}
\newcommand{\bt}{\boxtimes}

\renewcommand{\1}{\mathds{1}}
\newcommand{\oeq}{\stackrel{\otimes}{\cong}}
\newcommand{\breq}{\stackrel{\operatorname{br}}{\cong}}
\newcommand{\FPdim}{\operatorname{FPdim}}
\newcommand{\pu}{\operatorname{pu}}
\newcommand{\ord}{\operatorname{ord}}

\begin{document}

\title{Generalized Witt and Morita equivalences}
%\author{Liang Kong, Dmitri Nikshych, Yilong Wang, Hao Zheng}
%\date{}
\author[a]{Liang Kong \thanks{Email: \href{mailto:kongl@sustech.edu.cn}{\tt kongustc@outlook.com}}}
% \author[d]{Dmitri Nikshych \thanks{Email: \href{mailto:dmitri.nikshych@unh.edu
% }{\tt dmitri.nikshych@unh.edu}}}
\author[b]{Yilong Wang \thanks{Email: \href{mailto:wyl@bimsa.cn}{\texttt wyl@bimsa.cn}}}
\author[c,b]{Hao Zheng \thanks{Email: \href{mailto:haozheng@mail.tsinghua.edu.cn}{\tt haozheng@mail.tsinghua.edu.cn}}}
%\affil[a]{Shenzhen Institute for Quantum Science and Engineering, \authorcr Southern University of Science and Technology, Shenzhen, 518055, China}
\affil[a]{International Quantum Academy, Shenzhen 518048, China}
%\affil[c]{Guangdong Provincial Key Laboratory of Quantum Science and Engineering, \authorcr Southern University of Science and Technology, Shenzhen, 518055, China}
%\affil[d]{Department of Mathematics and Statistics, \authorcr University of New Hampshire, Durham, NH 03824, USA}
\affil[b]{Beijing Institute of Mathematical Sciences and Applications, Beijing 101408, China}
\affil[c]{Institute for Applied Mathematics, Tsinghua University, Beijing 100084, China}
\date{\vspace{-5ex}}

\maketitle

\begin{abstract}
In this work, we introduce a family of new equivalence relations among fusion categories that are less refined than the usual Morita equivalence. We obtain abelian groups by quotienting these new equivalence relations from the commutative monoids of the equivalence classes of all fusion categories. Moreover, we upgrade them to equivalence relations among nondegenerate braided fusion categories that are more refined than the usual Witt equivalence. As a consequence, we obtain new abelian groups that are more refined than the usual Witt group. These new groups allow us to access the internal structures within Witt classes. We expect that they are useful in the classification program of (braided) fusion (higher) categories and in the study of gapless edges of 2+1D topological orders. 
\end{abstract}

\tableofcontents

\section{Introduction}

The classification of (braided) fusion categories is a well-established mathematical program, which has been developed for years and has accumulated a lot of interesting and important results. The motivation of this program is not limited to the natural interests and importance of understanding the intrinsic structures of (braided) fusion categories and their representations. It has also been motivated by the classification programs in many other branches of mathematics that are deeply intertwined with (braided) fusion categories. Important examples include (but are not limited to) vertex operator algebras, (weak) Hopf algebras, infinite-dimensional Lie algebras, subfactors, conformal nets, quantum invariants in low dimensional topology, and the mathematical theory of topological/conformal field theories. 

In recent years, there have been increasing demands from physics for further advances in this program, especially in the study of topological orders \cite{Wen89,Wen90} in condensed matter physics (see \cite{Wen17,Wen19} for recent reviews). More precisely, a modular tensor category (MTC), i.e., a braided fusion category with additional structures and properties, gives a precise mathematical description of a 2+1D topological order up to chiral central charges (see for example \cite{Kit06}). A fusion category gives a precise mathematical description of a gapped boundary of a 2+1D topological order \cite{KK12}. Therefore, the classification of MTCs (or fusion categories) is essentially the same thing as that of 2+1D topological orders (or their gapped boundaries).

%\item From a very modern point of view, a fusion category can be viewed as a finite type generalized symmetry in a 1+1D non-topological quantum field theory \cite{FFRS06} or that in a 1+1D  topological quantum field theory \cite{DKR11,}. Therefore, the classification of fusion categories can be viewed as the modern version of the famous program of the classification of finite simple groups. 

%plays very important roles in the study of topological phases of matter and quantum field theory in general. More precisely, it gives a complete characterization of 1+1D anomalous topological orders; and it also defines the non-invertible symmetries for general quantum field theories in 1+1D. Therefore, to classify fusion category becomes an important questions in both mathematics and physics. 

There are at least two approaches to the classification problem. We briefly review the ideas 
% and the results 
of both approaches. This work focuses on the second approach. 

The first approach is to classify (braided) fusion categories with certain nice or natural properties or by fixing certain sub-structures or invariants, such as the rank (the number of isomorphism classes of simple objects), the (Frobenius-Perron) dimension, or the fusion ring. Many such classification programs have been actively conducted, and, sometimes, have been developed based on various types of sources for fusion categories, including the representation categories of finite groups, semisimple quasi-Hopf algebras, quantum groups at roots of unity, and subfactor theory. For example, when the Frobenius-Perron dimension of every object in a fusion category is an integer, such a fusion category is called integral. It is well-known that an integral fusion category is equivalent to the representation category of a semisimple quasi-Hopf algebra. Hence, it is natural to study such fusion categories by their dimensions and by taking advantage of the techniques developed in the study of finite groups and Hopf algebras (see for example \cite{EGO04,ENO11,JL09}). 

Unlike fusion categories coming from other sources, subfactor theory is capable of producing ``exotic'' fusion categories that are non-(weakly) integral and have noncommutative fusion rules, such as the ones given by the (extended) Haagerup subfactors \cite{AH99, BPSS12}. Moreover, subfactors have universality in the following sense: any unitary fusion category can be associated with some (essentially unique) hyperfinite subfactor such that the even half of its standard invariant is the given fusion category \cite{FalRau13}. Therefore, understanding finite-index subfactors is crucial in the classification of general fusion categories. The classification of subfactors of the small index was pioneered in \cite{GHJ89,Haa93}, and the classification of subfactors up to index 5 and beyond is an accumulation of works of many, see \cite{JMS14, AMP23} for a summary.

One can also start with fusion rings, which are based rings sharing similar formal properties as the Grothendieck rings of fusion categories, and ask whether there are fusion categories realizing (or categorifying) such rings as their Grothendieck rings. The Tambara-Yamagami categories \cite{Tambara1998} are the perfect examples of products of such a classification. The Tambara-Yamagami categories are now understood as a special case of the near-group fusion categories, which are fusion categories with all but 1 simple object invertible \cite{Sieh03,EvG14,Izu17}. Although the existence of many such fusion categories and their conjectural properties and generalizations are still open, they have already inspired a lot of related research including \cite{GI20, IT21, CIP21}. Analytic approaches have also been employed in the categorification problem of fusion rings, resulting in necessary (computationally efficient) conditions for a fusion ring to admit unitary categorification, see \cite{LPW21,LPR22}. Classification of fusion categories by rank is also studied in \cite{Ost03f, O15}. Results on categorifiable integral fusion rings can also be seen in the literature \cite{rings-rk5}, but without additional structures or assumptions, it would be exponentially more difficult to classify general fusion categories of higher ranks.

The notion of a modular tensor category (MTC) can be regarded as a categorical generalization of that of a nondegenerate quadratic form on a finite abelian group. From this point of view, it is fair to say that the classification of MTCs and the associated topological field theories is similar to Wall's classification of quadratic forms \cite{Wall1963}, which was used in the study of differential topology \cite{Wal63-1}.

Several classification programs of MTCs have been actively pursued. One of the theoretical foundation of such classification programs is the rank finiteness theorem \cite{RankFinite}, which says up to equivalence, there are only finitely many MTCs of a given rank. The classification of MTCs with ranks up to 5 is finished in \cite{RSW09, Rank5}, and partial results on the classification of MTCs of rank 6 is given in \cite{Cre18, NRWW23}. In the above classifications, the action of the absolute Galois group on the set of simple objects is utilized as the key tool, which inspired the classification of MTCs according to the number of Galois orbits. In particular, MTCs with a single Galois orbit are completely classified in \cite{NWZ20a}, and partial results on MTCs with two Galois orbits are obtained in \cite{PSYZ23}. In the development of these results, properties of the modular data and congruence representations of $\operatorname{SL}_2(\mathbb{Z})$ associated with MTCs have played such prominent roles that classifying modular data that is potentially realizable by MTCs itself has become an active area of research in the classification program \cite{NRWW23, NRW23}. Such techniques are also applied to the study of super-modular categories \cite{CKSY23}. Similar to the fusion case, the classification of MTCs by dimension is also an important direction, for example, \cite{BR12, IntModFP, DonNat18, CzP22} and \cite{ABPP23}.

The second approach is to organize the entire classification program in multiple layers. By introducing an equivalence relation among fusion categories, one can first try to work on a rough classification up to this equivalence relation. Two well-known examples of  this idea are the Morita equivalence among fusion categories \cite{Mug031} and the Witt equivalence among nondegenerate braided fusion categories \cite{DMNO} (see \cite{DNO,Sch17,NRWZ22} for further developments). The study of the Morita classes of fusion categories and that of the Witt classes of nondegenerate braided fusion categories are deeply intertwined as we explain below. Along the way, we also explain the mathematical and physical motivations to go beyond the Morita and the Witt equivalences. 
\bnu
\item 
The set of Morita classes of fusion categories is a commutative monoid under the Deligne tensor product $\boxtimes$. However, it is not clear if there are any additional and useful structures on this monoid. It is natural to ask if it is possible to quotient more so that we obtain an abelian group. In other words, it suggests us to find new equivalence relations among fusion categories that are less refined than the Morita equivalence. 
\item 
It is well known that two fusion categories are Morita equivalent if and only if they share the same Drinfeld center \cite{Mug031,ENO11}. It means that to classify fusion categories up to Morita equivalence amounts to classifying nondegenerate braided fusion categories within the trivial Witt class. It demands us to know the internal structures within the trivial Witt class. In other words, it suggests us to find new equivalence relations among nondegenerate braided fusion categories that are more refined than the Witt equivalence relation. 
\item 
According to the unified mathematical theory of gapped and gapless edges of 2+1D topological orders \cite{KZ18,KZ20,KZ21-gapless-2, CJKYZ20,CW23,LY23,CJW22}, in order to study the phase diagram of the gapless edges of a fixed 2+1D bulk phase, one needs to unravel the hidden structures within each Witt class, such as the orbits under the action of the commutative monoid of elements within the trivial Witt class. 
%\red{those lying in the action of the trivial Witt class (as a commutative monoid) on a non-trivial Witt class. [This sentence is very strange. According to our discussions, you wanted something like the following: ... and the actions of the trivial Witt class as a commutative monoidal on such structures.]} Unfortunately, the Witt group does not provide any useful information of this action. 
This suggests that we should find some refinements of the Witt group. 
\enu

In this work, inspired by an earlier physical proposal in \cite[Section\ VIII.D \& VIII.E]{KW14}, we introduce a family of equivalence relations among fusion categories so that we obtain abelian groups if we quotient these equivalence relations. Moreover, we upgrade them to equivalence relations among nondegenerate braided fusion categories such that they give refinements of the usual Witt equivalence. We expect that such obtained new groups will be useful to the classification program of (braided) fusion categories and reveal the internal structures of each Witt class that are important to physical applications. %in the study of the phase diagram of the gapless boundaries of 2+1D topological orders. 

\bigskip
\noindent {\bf Acknowledgements}: We would like to thank Dmitri Nikshych, Xiao-Gang Wen and Zhi-Hao Zhang for useful comments. LK is supported by NSFC (Grant No. 12574175) and by Guangdong Basic and Applied Basic Research Foundation (Grant No. 2020B1515120100). 
%DN is supported  by  the  US National  Science  Foundation  under  grant DMS-2302267.
YW is supported by NSFC (Grant No. 12301045) and by Beijing Natural Science Foundation Key Program (Grant No. Z220002) and by the BIMSA startup fund.
HZ is supported by NSFC under Grant No. 11871078 and by Startup Grant of Tsinghua University and BIMSA.

\section{Preliminaries}
In this section, we briefly review basic concepts related to (braided) fusion categories and set up notations and conventions (see \cite{EGNO, DGNO-BFC} for more details). 

\subsection{Fusion categories}
A fusion category $\CX$ over $\mathbb{C}$ is a finitely semisimple, $\mathbb{C}$-linear abelian, rigid monoidal category whose tensor unit $\1$ is simple. The category $\CX$ equipped with the opposite tensor product $a \overline{\otimes} b := b \otimes a$ for $a, b \in \CX$ is also a fusion category, which is denoted by $\overline{\CX}$. We denote the set of isomorphism classes of $\CX$ by $\Irr(\CX)$. 

For all $a \in \CX$, we fix a choice of its left dual, and denote it by $(a^*, \ev_a: a^* \otimes a \to \1, \coev_a: \1 \to a \otimes a^*)$ with $a^* \in \CX$. A simple object $X \in \Irr(\CX)$ is called \emph{invertible} if both $\ev_X$ and $\coev_X$ are isomorphisms, and $\CX$ is called pointed if all of its simple objects are invertible. It is not hard to derive that, for any pointed fusion category $\CX$, there exists a finite group $G$ and a normalized 3-cocycle $\omega \in Z^3(G, \mathbb{C}^\times)$ such that $\CX$ is tensor equivalent to $\vect_{G}^{\omega}$, the category of finite-dimensional $G$-graded complex vector space whose associativity is twisted by $\omega$ (see also \cite{EGNO}).

For any morphism $f: a \to a^{**}$, we define its left quantum trace to be  $\tr_{a}(f) = \ev_{a^*} \circ (f \otimes \id_{a^*}) \circ \coev_{a} \in \End_{\CX}(\1) \cong \mathbb{C}$. According to \cite{Mug031, ENO05}, since $\CX$ is fusion, $a \cong a^{**}$ for all $a \in \CX$, and the squared norm of a simple object $X \in \Irr(\CX)$ is defined to be $|X|^2 = \tr_{X}(h) \cdot \tr_{X^*}((h^{-1})^*) \in \mathbb{C}$, where $h: X \to X^{**}$ is any nonzero morphism. Then the \emph{global dimension} of $\CX$ is defined to be 
\begin{equation*}
\dim(\CX) = \sum_{X \in \Irr(\CX)} |X|^2\,.  
\end{equation*}
The Frobenius-Perron Theorem provides another notion of dimension on $\CX$: For each object $a \in\CX$, its Frobenius-Perron dimension (FP-dimension) $\FPdim(a)$ is the largest positive eigenvalue of the matrix of (left) multiplication by $a$. Then the FP-dimension of $\CX$ is defined to be $\FPdim(\CX) = \sum_{X \in \Irr(\CX)} \FPdim(X)^2$. It is shown in \cite{ENO05} that both $\dim(\CX)$ and $\FPdim(\CX)$ are totally positive algebraic integers, i.e., any of the Galois conjugates of these dimensions is a positive real number. Moreover, $\FPdim(\CX) \ge \dim(\CX)$, and when $\FPdim(\CX) = \dim(\CX)$, we call $\CX$ \emph{pseudounitary}.

A braided fusion category is a fusion category $\CA$ equipped with a braiding, i.e., natural isomorphisms $\beta_{a, b}: a \otimes b \xrightarrow{\cong} b \otimes a, \forall a,b\in \CA$ satisfying the hexagon axioms \cite{JS93}. In particular, $\beta$ endows the identity functor with a monoidal structure yielding a monoidal equivalence $\CA \stackrel{\otimes}{\cong} \overline{\CA}$. The M\"uger center of $\CA$, denoted by $\FZ_2(\CA)$, is the full fusion subcategory of $\CA$ generated by objects $a \in \CA$ such that $\beta_{b,a} \circ \beta_{a, b} = \id_{a \otimes b}$ for all $b \in \CA$. A braided fusion category $\CA$ is \emph{symmetric} if $\FZ_2(\CA) = \CA$. By Deligne's results \cite{Del90,Del02}, if $\CA$ is a symmetric fusion category, then there exists a finite group $G$ such that the underlying fusion category of $\CA$ is tensor equivalent to $\rep(G)$, the category of finite-dimensional complex representations of $G$, while the braiding on $\CA$ is either the usual one on $\rep(G)$, in which case we write $\CA \breq \rep(G)$; or is twisted by a central element $z \in G$ of order 2, in which case we write $\CA \breq \rep(G,z)$.

If the M\"uger center of $\CA$ is minimal, i.e., $\Irr(\FZ_2(\CA)) = \{\1\}$, then we call $\CA$ \emph{nondegenerate}. In other words, $\CA$ is nondegenerate if and only if there is a braided tensor equivalence between $\FZ_2(\CA)$ and $\vect$, the category of finite-dimensional complex vector spaces. The Drinfeld center construction \cite{JS91,Maj91} assigns a braided fusion category $\FZ_1(\CX)$ to any given fusion category $\CX$. It is well-known that $\FZ_1(\CX)$ is nondegenerate \cite{Mug032,DGNO-BFC}, or equivalently, we have a braided equivalence $\FZ_2(\FZ_1(\CX)) \breq \vect$. Moreover, we have $\dim(\FZ_1(\CX)) = \dim(\CX)^2$ \cite[Thm.\ 1.2]{Mug032} and $\FPdim(\FZ_1(\CX)) = \FPdim(\CX)^2$ \cite[Prop.\ 8.12]{EGNO}.

\subsection{2-Categories and commutative monoids} \label{sec:sub-monoids}

We discuss some symmetric monoidal 2-categories and the associated commutative monoids and introduce some notations and conventions along the way. 
\bnu
\item We denote by $\Fus$ the 2-category of fusion categories (as objects), monoidal equivalences (as 1-morphisms) and invertible monoidal natural transformations (as 2-morphisms). It is symmetric monoidal with the tensor unit given by $\vect$ and the tensor product given by the Deligne tensor product $\boxtimes$. We denote by $\fus$ the underlying set of equivalence classes of objects in $\Fus$. It is clear that $\fus$ is a commutative monoid with the identity given $\vect$ and the multiplication given by $\boxtimes$. We introduce a few monoidal sub 2-categories of $\Fus$ and the associated submonoids of $\fus$. 

\bnu
\item We denote by $\cfrFus$ the monoidal sub-2-category of $\Fus$ consisting of those fusion categories that have commutative fusion rings. We denote its underlying commutative monoid by $\cfrfus$. 

\item We denote by $\bfcFus$ the monoidal sub-2-category of $\Fus$ consisting of those fusion categories that are monoidally equivalent to some braided fusion categories. We denote its underlying commutative monoid by $\bfcfus$.

\item We denote by $\nbfcFus$ the monoidal sub-2-category of $\Fus$ consisting of those fusion categories that are monoidally equivalent to some nondegenerate braided fusion categories. We denote its underlying commutative monoid by $\nbfcfus$.

\item We denote by $\sfcFus$ the monoidal sub-2-category of $\Fus$ consisting of those fusion categories that are monoidally equivalent to some symmetric fusion categories. We denote its underlying commutative monoid by $\sfcfus$.
\enu
Let $S$ be a subset of $\fus$. For convenience, from now on, we use the notation $\CA \in S$ to mean a fusion category $\CA$ whose monoidal equivalence class is in $S\subset \fus$.

\item We denote by $\BFC$ the symmetric monoidal 2-category of nondegenerate braided fusion categories (as objects), braided monoidal equivalences (as 1-morphisms) and invertible braided monoidal natural transformations (as 2-morphisms). The tensor unit of $\BFC$ is again $\vect$ and the tensor product in $\BFC$ is again the Deligne tensor product $\boxtimes$. We denote its underlying commutative monoid by $\bfc$. We discuss a couple of monoidal sub-2-categories of $\BFC$ and their underlying commutative submonoids of $\bfc$ below. 

\bnu
\item We denote by $\NBFC$ the symmetric monoidal 2-category of nondegenerate braided fusion categories (as object), braided monoidal equivalences (as 1-morphisms) and invertible braided monoidal natural transformations (as 2-morphisms). We denote its underlying commutative monoidal by $\nbfc$. 

\item We denote by $\SFC$ the symmetric monoidal 2-category of symmetric fusion categories (as object), braided monoidal equivalences (as 1-morphisms) and invertible braided monoidal natural transformations (as 2-morphisms). We denote its underlying commutative monoidal by $\sfc$. 

\enu
Let $T$ be a subset of $\bfc$. For convenience, from now on, we use the notation $\CA \in T$ to mean a braided fusion category $\CA$ whose braided monoidal equivalence class is in $T\subset \bfc$.
\enu

\begin{rem}
Note that the forgetful functor $\BFC \to \Fus$ induces surjective monoidal maps: 
$$
\bfc \twoheadrightarrow \bfcfus, \quad \nbfc \twoheadrightarrow \nbfcfus, \quad 
\sfc \twoheadrightarrow \sfcfus.
$$ 
The Drinfeld center construction defines a monoid map $\fus \to \nbfc$. 
% \cite{Mug032,DGNO-BFC}. 
\end{rem}

The classification of fusion categories amounts to a study of the commutative monoid $\fus$. Our approach is inspired by the two important equivalence relations among fusion categories characterized by Drinfeld centers, namely, the Morita equivalence of fusion categories and the Witt equivalence of nondegenerate braided fusion categories. We recall these notions below. 
\begin{itemize}
\item 
Two fusion categories $\CX,\CY \in \Fus$ are called \emph{Morita equivalent} (denoted by $\CX \mrt \CY$), if there exist a semisimple $\CX$-$\CY$-bimodule category $\CP$ and a semisimple $\CY$-$\CX$-bimodule category $\CQ$ such that 
$$
\CP \boxtimes_\CY \CQ \simeq \CX, \quad\quad \CQ \boxtimes_\CX \CP \simeq \CY
$$
as bimodule categories. It was known that $\CX \mrt \CY$ if and only if $\FZ_1(\CX) \breq \FZ_1(\CY)$ \cite{Mug031,ENO11}. 

\item 
Two nondegenerate braided fusion categories $\CA, \CB \in \NBFC$ are called \emph{Witt equivalent} \cite{DMNO} (denoted by $\CA \weq \CB$), if there exist $\CX$, $\CY \in \fus$ such that 
$$
\CA \boxtimes \FZ_1(\CX) \breq \CB \boxtimes \FZ_1(\CY)\,.
$$
We denote the Witt equivalence class of $\CA \in \nbfc$ by $[\CA]$. The set of Witt equivalence classes form an abelian group $\Witt$, called the Witt group. The Witt group $\Witt$ of nondegenerate braided fusion categories generalizes (in fact, contains) the classical Witt group of nondegenerate quadratic forms on finite abelian groups \cite{DGNO-BFC}.
\end{itemize}

\begin{rem}
We give a brief remark to our notations. The notation ``$\weq$'' for the Witt equivalence is motivated by the fact that the Witt equivalence can be viewed as some kind of ``2-Morita equivalence''. More precisely, the delooping of a braided fusion category $\CA$, denoted by $\Sigma\CA$, can be  identified with $\RMod_\CA(2\vect)$ \cite{DR18,GJF19}. For $\CA,\CB \in \NBFC$, $\CA \weq \CB$ if and only if $\Sigma\CA$ is Morita equivalent to $\Sigma\CB$, i.e., $\Sigma^2\CA = \RMod_{\Sigma \CA}(3\vect) \simeq \RMod_{\Sigma \CB}(3\vect) = \Sigma^2\CB$ as 3-categories \cite{JF22,KZ22}. This means that the Witt equivalence is indeed a higher Morita (or 2-Morita) equivalence. Moreover, $\Sigma^2$ assigns a nondegenerate braided fusion 1-category to an invertible separable 3-category, which is precisely a simple invertible object in $4\vect$ \cite{JF22,KZ22}. Therefore, the Witt group is precisely the group of invertible objects in $4\vect$. 
\end{rem}

\begin{rem}
It is worth mentioning that pairs of Morita equivalent unitary fusion categories can be obtained naturally from finite-index subfactors of finite depth, and vice-versa. See \cite{JMS14} for a survey and further references. 
\end{rem}

Let $S$ be a nonempty subset of $\fus$ that is closed under the multiplication $\boxtimes$. In Sections \ref{sec:eq-rel} and \ref{sec:gp-str}, we study the equivalence relations on $\fus$ defined by $S$. Note that one can always add the identity of $\fus$ to $S$ to upgrade it to a submonoid of $\fus$. 
\begin{expl}
We list a handful of submonoids of $\fus$ for future use. 

\bnu
\item Submonoids of $\bfc$.

\bnu
\item A submonoid of $\nbfc$  that is of particular interest is $S_0 \subset \nbfc$ consisting of braided monoidal equivalence classes of objects in $\NBFC$ that are braided equivalent to $\FZ_1(\CX)$ for some $\CX\in\Fus$. By an abuse of notation, we simply write
\begin{equation}
  \label{eq:1}
  S_0 := \FZ_1(\fus) = \{\FZ_1(\CX) \mid \CX \in \fus\}\,.
\end{equation}
Note that $[S_0] := \{[\vect]\}$ is the trivial subgroup of the Witt group $\Witt$. 

\item 
More generally, for any Witt subgroup $W' \subset \Witt$, one can consider its preimage in $\nbfc$, i.e., the submonoid 
$\{\CA \in \nbfc \mid [\CA] \in W'\}$.

\item Submonoids of $\sfc$ such as $\{\rep(G)\mid G \text{ is a finite group}\}$.
\enu

\item Submonoids of $\fus$ that are not contained in $\bfc$ are also worth studying.

\bnu
\item Since there are fusion categories with commutative fusion rules that admit no braiding (such as $\vect_{\BZ/3\BZ}^{\omega}$ with $\omega$ belonging to a nontrivial cohomology class), we have proper inclusions 
$$
\nbfcfus \subsetneq \bfcfus \subsetneq \cfrfus \subsetneq \fus\,.
$$

\item For a fixed prime $p$, the collection of all pointed fusion categories whose underlying group of invertible objects is a $p$-group;

\item Group theoretical fusion categories, i.e., fusion categories which are Morita equivalent to pointed fusion categories.

\item Other submonoids of $\fus$ characterized by dimensions, such as the submonoid of weakly-integral fusion categories, which is contained in the submonoid of pseudounitary fusion categories \cite[Prop.\ 8.24]{ENO05}.
\enu
\enu
\end{expl}

\section{Equivalence relations on (braided) fusion categories}\label{sec:eq-rel}

\subsection{Generalization of the Witt equivalence}\label{subsec:witt}
Let $S \subset \fus$ be a non-empty subset that is closed under the multiplication $\boxtimes$. We generalize the Witt equivalence on $\nbfc$ using $S$ as follows. The idea of generalizing the Witt equivalence first appeared in the physics literature \cite[Sec.\ VIII.D \& VIII.E]{KW14} and was formulated in physical language. 

\begin{defn}
Two nondegenerate braided fusion categories $\CA$ and $\CB$ are called \emph{$S$-Witt equivalent}, denoted by $\CA \weq_S \CB$, if there exist two fusion categories $\CP, \CQ\in S$ such that
\be \label{eq:S-Witt-eq-relation}
\CA \boxtimes \FZ_1(\CP) \breq \CB \boxtimes \FZ_1(\CQ).
\ee
\end{defn}

\begin{rem}
When $S=\fus$, the $\fus$-Witt equivalence is precisely the usual Witt equivalence. The $S$-Witt equivalence for $S=\nbfcfus$ is the mathematical reformulation of \cite[Def.\ 23]{KW14}. 
\end{rem}

\begin{lem}\label{lem:s-rel}
This relation $\weq_S$ is an equivalence relation on $\nbfc$, i.e., 
\bnu
\item[(1)] $\CA \weq_S \CA$; 
\item[(2)] $\CA \weq_S \CB$ implies that $\CB \weq_S \CA$; 
\item[(3)] $\CA \weq_S \CB$ and $\CB \weq_S \CC$ imply $\CA \weq_S \CC$. 
\enu
\end{lem}
\pf
(1) and (2) hold obviously. By the assumption of (3), we have $\CA \boxtimes \FZ_1(\CP) \breq \CB \boxtimes \FZ_1(\CQ)$ and $\CB \boxtimes \FZ_1(\CX) \breq \CC \boxtimes \FZ_1(\CY)$ for some $\CP,\CQ,\CX,\CY \in \fus$. Then we have
$$
\CA \boxtimes \FZ_1(\CP) \boxtimes \FZ_1(\CX) \breq \CB \boxtimes \FZ_1(\CQ) \boxtimes \FZ_1(\CX) \breq \CC \boxtimes \FZ_1(\CQ) \boxtimes \FZ_1(\CY), 
$$
which implies that $\CA \weq_S \CC$. 
\epf

We denote the $S$-Witt equivalence class, or simply the $S$-Witt class of $\CA$ by $[\CA]_S$. The set of $S$-Witt classes is denoted by $\Witt_S$, i.e., $\Witt_S=\nbfc /\hspace{-1mm} \weq_S$. Note that $\Witt_\fus = \Witt$ is the usual Witt group.

The following lemmas are direct consequences of the definitions.

\begin{lem}
The equivalence relation $\weq_S$ is compatible with $\boxtimes$. That is, if $\CA \weq_S \CA'$ and $\CB \weq_S \CB'$, then $\CA \boxtimes \CB \weq_S \CA' \boxtimes \CB'$. As a consequence, $\Witt_S$ is a commutative monoid. \qed
\end{lem}

\begin{lem}\label{lem:contain}
Let $S$ and $S'$ be nonempty subsets of $\fus$ that are closed under $\boxtimes$. If $S \subset S'$, then for all $\CA$, $\CB \in \nbfc$, $\CA\weq_{S} \CB$ implies $\CA \weq_{S'} \CB$. We have a canonical surjective monoid homomorphism $\Witt_S \twoheadrightarrow \Witt_{S'}$. \qed
\end{lem}

If $\CA \weq_{S} \vect$, then by \cite[Prop.\ 5.8]{DMNO}, $\CA \breq \FZ_{1}(\CF)$ for some fusion category $\CF$, where $\CF$ is not necessarily in $S$. However, the proof of \cite[Prop.\ 5.8]{DMNO} can be used to show that in certain cases we do have $\CF \in S$. For example, we have the following observations.

\begin{lem}
  Let $S_{\pu}\subset \fus$  be the submonoid consisting of pseudounitary fusion categories. For all $\CA \in \nbfc$, $\CA\weq_{S_{\pu}} \vect$ if and only if $\CA \breq \FZ_1(\CX)$ for some $\CX \in S_{\pu}$.
\end{lem}
\pf
If $\CA \weq_{S_{\pu}} \vect$, there exist pseudounitary fusion categories $\CB$, $\CC \in S_{\pu}$ such that $\CA \boxtimes \FZ_{1}(\CB) \breq \FZ_{1}(\CC)$. Then $\dim(\CA) = \frac{\dim(\CB)^2}{\dim(\CC)^2} = \frac{\FPdim(\CB)^2}{\FPdim(\CC)^2} = \FPdim(\CA)$. Moreover, by \cite[Prop.~5.8]{DMNO}, there exist a fusion category $\CX$ such that $\CA \breq \FZ_1(\CX)$. Then $\dim(\CX)$ and $\FPdim(\CX)$ are positive square roots of $\dim(\CA)$ and $\FPdim(\CA)$ respectively, so we have $\dim(\CX) = \FPdim(\CX)$, i.e., $\CX \in S_{\pu}$. Conversely, if $\CA \breq \FZ_1(\CX)$ for some $\CX \in S_{\pu}$, then $\CA \weq_{S_{\pu}} \vect$ by definition.
\epf

\begin{lem}
Let $S_{\gt} \subset \fus$ be the submonoid consisting of group-theoretical fusion categories. For all $\CA \in \nbfc$, then $\CA \weq_{S_{\gt}} \vect$ if and only if $\CA \breq \FZ_{1}(\CP)$ for some group-theoretical fusion category $\CP \in S_{\gt}$.
\end{lem}
\pf
It suffices to show that if $\CA\weq_{S_{\gt}}\vect$, then $\CA \breq \FZ_{1}(\CP)$ for some $\CP \in S_{\gt}$. By assumption, there exist group-theoretical fusion categories $\CB$, $\CC \in S_{\gt}$ such that $\CA \boxtimes \FZ_{1}(\CB) \breq \FZ_{1}(\CC)$. 
% Note that $\FZ_{1}(\CC)$ is group-theoretical as $\CC$ is. 
By the proof of \cite[Prop.~5.8]{DMNO}, there exists a connected \'etale algebra $A \in \FZ_1(\CC)$ and a fusion category $\CP$ such that $\CA \breq \FZ_{1}(\CC)_{A}^{0} \breq \FZ_{1}(\CP)$, where $\FZ_1(\CC)_A^0$ stands for the category of local $A$-modules. So by \cite[Thm.\ 3.16]{DS17}, $\FZ_{1}(\CC)_{A}^{0} \breq \FZ_{1}(\CP)$ is braided equivalent to the Drinfeld center of a pointed fusion category, i.e., $\CP$ is group-theoretical.
\epf

\begin{rem}
The above lemma can also be derived from \cite[Prop.~9.7.9]{EGNO}, which essentially follows from the surjectivity of the free module functor (\cite[Sec.\ 3]{DMNO}) from a braided fusion category to the category of right modules of a connected \'etale algebra in the category.
\end{rem}

For $S\subsetneq \fus$, the canonical surjective monoid map $\Witt_S \twoheadrightarrow \Witt$ gives us a first look at the internal structure hidden in the problem of classifying all fusion categories. More precisely, two fusion categories are Morita equivalent if and only if they share the same Drinfeld center. In other words, the set of elements in the trivial Witt class gives the complete classification of the Morita classes of fusion categories. This inspires us to generalize the Morita equivalence on fusion categories.

\subsection{Generalization of the Morita equivalence}

\begin{defn}
Let $S\subset \fus$ be a nonempty subset closed under $\boxtimes $. Two fusion categories $\CL$ and $\CM$ are called \emph{$S$-Morita equivalent}, denoted by $\CL \mrt_S \CM$, if $\FZ_1(\CL) \weq_S \FZ_1(\CM)$, or equivalently, $\CL \mrt_S \CM$ if there exists $\CP,\CQ \in S \subset \fus$ such that 
$$
\CL \boxtimes \CP  \mrt  \CM \boxtimes \CQ\,.
$$ 
\end{defn}
When $S$ consists of only the unit of $\fus$, $\mrt_S$ is precisely the usual Morita equivalence $\mrt$.

\begin{lem}
The $S$-Morita equivalence is a well-defined equivalence relation.
\end{lem}

\begin{proof}
  This follows from the definition and Lemma \ref{lem:s-rel}.
\end{proof}

\begin{lem}\label{lem:embed}
The map $\fus/\hspace{-1mm}\mrt_S \, \to \Witt_S$ defined by sending the $S$-Morita equivalence class of any fusion category $\CX$ to $[\FZ_1(\CX)]_S$ is an injective homomorphism of commutative monoids.
\end{lem}
\begin{proof}
This follows directly from the definition.
\end{proof}

%\begin{rem} As we know, the classification of unitary fusion categories is equivalent to that of subfactors. So this weak Morita group should very useful in the study of subfactors. \end{rem}

% If we are only interested in the classification of fusion categories, we do not care if quotienting the equivalence relation gives a group for all NBFC's. We only care those equivalence relations such that their equivalence classes of those in the usual trivial Witt class becomes a group. 

%Let $\CL$ be an indecomposable multi-fusion category. After quotienting certain equivalence relation, if $\CL$ has a inverse, then its inverse is necessarily the dual of $\CL$. On the other hand, $\CL^\rev$ in $\Fus$ is the dual of $\CL$. It is natural to consider making $\CL^\rev$ the inverse of $\CL$ after the quotient. Due the following well-known fact \be \CL \boxtimes \CL^\rev \mrt \FZ_1(\CL). \ee 

Therefore, we can identify $\fus/\hspace{-1mm}\mrt_S$ with the commutative submonoid  $\{[\FZ_1(\CX)]_S\mid \CX \in \fus\} \subset \Witt_S$. We provide more details in the next section. 

\begin{rem}
Among many choices of $S$, there seems to be a natural one $S=S_0$ in Eq.\ \eqref{eq:1}. In particular, we show in Corollary \ref{cor:grp} that $\fus/\hspace{-1mm}\mrt_{S_0}$ is a group. We suspect that this group is useful in the program of classifying fusion categories.  
\end{rem}

\section{Group structure from \texorpdfstring{$S$}{}-Witt equivalence}\label{sec:gp-str}
In this section, we study conditions for $\Witt_S$ to be a group. 
\begin{lem} \label{lem:action_S}
Let $S \subset \fus$ be a nonempty subset closed under $\boxtimes $. If $\nbfc$ acts on $S$ by $\boxtimes$, i.e., $\boxtimes(\nbfc \times S)=S$, then $\Witt_S$ is an abelian group. 
\end{lem}
\pf
For $\CA \in \nbfc$ and $\CP \in S$, we have 
$\CA \boxtimes \CA^\rev \boxtimes \FZ_1(\CP) \breq \FZ_1(\CA\boxtimes \CP)$. Hence, $\CA \boxtimes \CA^\rev \weq_S \vect$. 
\epf

\begin{expl}
Note that $\nbfcfus$, $\bfcfus$, $\cfrfus$ and $\fus$ are examples of $S$ that are closed under the $\nbfc$-action. By Lemma\,\ref{lem:action_S}, we obtain abelian groups $\Witt_{\nbfcfus}, \Witt_{\bfcfus}$, $\Witt_{\cfrfus}$ and $\Witt_{\fus}=\Witt$. 
\end{expl}

%\begin{prob}Is it possible to find a sufficient and necessary condition for $\Witt_S$ to be a group? \end{prob}

Tautologically, $\Witt_S$ is a group if and only if there exists
\[\varphi_S: \nbfc \to \nbfc \times S\,,\quad \CA \mapsto (\varphi_S(\CA)_1, \varphi_S(\CA)_2)\] 
such that the induced map 
\[\tilde{\varphi}_S: \nbfc \to \nbfc\,, \quad \CA \mapsto \CA \boxtimes \varphi_S(\CA)_1 \boxtimes \FZ_{1}(\varphi_S(\CA)_2)\]
satisfies $\tilde{\varphi}_S(\nbfc) \subseteq \FZ_{1}(S)$.

\begin{lem}\label{lem:trival-center}
Let $S$ be a submonoid of $\fus$. If $\CX \in S$, then $\FZ_1(\CX) \weq_S  \vect$.
\end{lem}
\begin{proof}
By assumption, $\vect \in S$, so $\FZ_1(\CX) \bt \FZ_1(\vect) \breq \vect \bt \FZ_1(\CX)$.  
\end{proof}

%By the above lemma, if $S\subset\fus$ is a submonoid, then a necessary condition for $\Witt_S$ to be a group is that for any $\CA \in \nbfc$, $\CA \boxtimes \varphi_S(\CA)_1$ is the center of some fusion category, as $\CA \boxtimes \varphi_S(\CA)_1 \weq_S \vect$ implies $\CA \boxtimes \varphi_S(\CA)_1 \weq \vect$. 

Let $S \subset \fus$ be a submonoid. By Lemma\,\ref{lem:contain}, we obtain a short exact sequence of commutative monoids:
\begin{align*}
0 \to \ker(\omega_S) \to \Witt_S &\xrightarrow{\omega_S} \Witt \to 0 \\
[\CA]_S &\mapsto [\CA]
\end{align*}
where $[\CA]$ denotes the Witt class of $\CA$. Note that the submonoid
\[
\ker(\omega_S) = \{\, [\FZ_1(\CX)]_S\, |\, \mbox{$\CX$ is a fusion category}\} \cong (\fus/\hspace{-0.1cm}\mrt_S)
\]
encodes the information of the hidden structure of the trivial Witt class $[\vect]$. 

%by Lemma \ref{lem:embed}.

%This information might be important to the physics of topological phase transitions. \red{say more?}

The main theorem of this section is the following list of equivalent conditions for $\Witt_S$ to be a group.

\begin{thm}\label{thm:witt-s-group}
Let $S$ be a submonoid of $\fus$. The following statements are equivalent.
\bnu
\item[(1)] For all $\CA \in \nbfc$, $[\CA]_S$ has an inverse, and $[\CA]^{-1}_S = [\CA^{\rev} \bt \FZ_1(\CX)]_S$ for some $\CX \in \fus$.
\item[(2)] $\Witt_S$ is a group.
\item[(3)] $\ker(\omega_S)$ is a group.
\item[(4)] for all $\CX \in \fus$, there exists $\CY \in \fus$ and $\CP, \CQ \in S$ such that $\CX \boxtimes \CY \boxtimes \CP \mrt \CQ$.
\item[(5)] for all $\CX \in \fus$, there exists $\CY \in \fus$ and $\CQ \in S$ such that $\CX \boxtimes \CY \mrt \CQ$.
\enu
\end{thm}
\begin{proof}
(1)$\Rightarrow$(2)$\Rightarrow$(3) and (4)$\Rightarrow$(5) are obvious.

(3)$\Rightarrow$(4). Assume (3), then for all $\CX \in \fus$, the inverse of $[\FZ_1(\CX)]_S \in \ker(\omega_S)$ is of the form $[\FZ_1(\CY)]_S$ for some $\CY \in \fus$, i.e., $\FZ_1(\CX) \bt \FZ_1(\CY) \weq_S \vect$. By definition, this means there exist $\CP, \CQ \in S$ such that 
\begin{equation*}
\FZ_{1}(\CX) \boxtimes \FZ_{1}(\CY) \boxtimes \FZ_{1}(\CP) \breq \FZ_{1}(\CQ)\,, \end{equation*}
which implies $\CX \bt \CY \bt \CP \mrt \CQ$.

(5)$\Rightarrow$(1). For all $\CA \in \nbfc$, the nondegeneracy of the braiding of $\CA$ implies that
\[\omega_S([\CA \bt \CA^{\rev}]_S) = [\FZ_1(\CA)] = [\vect]\,,\]
which means $[\CA \bt \CA^{\rev}]_S \in \ker(\omega_S)$. So there exists $\CY \in \fus$ such that $\CA \bt \CA^{\rev} \weq_S \FZ_1(\CY)$. 
By (5) there exists $\CX \in \fus$ and $\CQ \in S$ such that 
\[\CA \bt \CA^{\rev} \bt \FZ_1(\CX) \weq_S \FZ_1(\CX \bt \CY) \breq \FZ_1(\CQ) \weq_S \vect\,,\]
where the $S$-Witt equivalence on the right hand side follows from Lemma \ref{lem:trival-center}.  Therefore, we have $[\CA]_S \cdot [\CA^{\rev} \bt \FZ_1(\CX)]_S = [\vect]_S$, and this completes the proof.
\end{proof}

%Recall we have defined $S_0 = \FZ_{1}(\fus)$.

\begin{cor}\label{cor:grp}
  Let $S$ be a submonoid of $\fus$ such that $S_0 \subset S$ (recall (\ref{eq:1}). Then $\ker(\omega_S)$ and $\Witt_S$ are both groups.  Moreover, for any $\CX \in \fus$, $[\FZ_1(\CX)]_S^{-1} = [\FZ_1(\overline{\CX})]_S = [\FZ_1(\CX)^{\rev}]_S \in \ker(\omega_S)$. In particular, $\ker(\omega_{S_0})$ is a group.
\end{cor}
\begin{proof}
For all $\CX \in \fus$, we have $\CX \bt \overline{\CX} \mrt \FZ_1(\CX) \in S_0 \subset S$, so we are done by Theorem \ref{thm:witt-s-group}.
\end{proof}

\begin{defn}
We say that a submonoid $S\subset\fus$ is {\em saturated}, if $\FZ_1(\CX)\weq_S\vect$ implies $\CX\in S$.
\end{defn}

In particular, if $S \subsetneq \fus$ and $S$ is saturated, then $S$-Witt equivalence is a stronger condition than the usual Witt equivalence. 

\begin{cor}
Let $S$ be a saturated submonoid of $\fus$. Then $\Witt_S$ is a group if and only if for all $\CX\in\fus$, there exists $\CY\in\fus$ such that $\CX\boxtimes\CY\in S$.
\end{cor}
\begin{proof}
  The statement follows from Theorem \ref{thm:witt-s-group} and the definition.
\end{proof}

\begin{prop}
Let $S$ be a submonoid of $\fus$. There exists a saturated submonoid $\bar S\subset\fus$ such that $\weq_S = \weq_{\bar S}$. Moreover, we have $S'\subset\bar S$ for any submonoid $S'\subset\fus$ such that $\weq_S = \weq_{S'}$. We refer to $\bar S$ as the {\em saturated closure} of $S$. 
\end{prop}
\begin{proof}
Let $\bar S = \{\CX\in\fus \mid \FZ_1(\CX)\weq_S\vect \}$. Note that $S\subset\bar S$ by Lemma \ref{lem:trival-center}. If $\CA\weq_{\bar S}\CB$, i.e. $\CA\boxtimes\FZ_1(\CP) \breq \CB\boxtimes\FZ_1(\CQ)$ where $\CP,\CQ\in\bar S$, then $\FZ_1(\CP) \weq_S \FZ_1(\CQ) \weq_S \vect$ hence $\CA \weq_S \CB$. This shows that $\weq_S = \weq_{\bar S}$. In particular, if $\FZ_1(\CX) \weq_{\bar S} \vect$, then $\FZ_1(\CX) \weq_S \vect$ hence $\CX\in\bar S$. This shows that $\bar S$ is saturated.

Assume $\weq_{S'} = \weq_S$. If $\CX\in S'$, then $\FZ_1(\CX) \weq_S \vect$ by Lemma \ref{lem:trival-center} hence $\CX\in\bar S$. That is, $S'\subset\bar S$.
\end{proof}

\begin{prob}
(1) Find the saturated closure of $\sfcfus$, $\nbfcfus$ and $\bfcfus$.

(2) Is $S_0$ contained in every saturated submonoid $S\subset\fus$ such that $\Witt_S$ is a group?
\end{prob}

\begin{expl}
  By \cite{DNO}, the usual Witt group $\Witt$ has no odd torsion: the order an element in $\Witt$ is either a power of 2 which is less than or equal to 32, or infinity. On the contrary, there exists $S$ such that $\Witt_S$ has elements of finite odd order. For instance, let $n\ge2$ be an integer and let $\fus^n = \{\CX^{\boxtimes n}\mid\CX\in\fus\}$. Then $\fus^n$ is a submonoid of $\fus$ and $\Witt_{\fus^n}$ is a group by Theorem \ref{thm:witt-s-group}. Moreover, it is easy to see that $g^n=1$ for all $g\in\ker(\omega_{\fus^n})$, so when $n$ is odd, any nontrivial element in $\ker(\omega_{\fus^n})$ has odd order.

Here we give an explicit example. Let $S = \fus^3$ and $p$ a prime. Then for $\CX = \rep(\mathbb{Z}/p \mathbb{Z})$, we have $[\FZ_1(\CX)]_S \ne [\vect]_S$. Indeed, otherwise there would exist $\CL$, $\CM \in \fus$ such that $\FZ_1(\CX) \boxtimes \FZ_1(\CL)^{\boxtimes 3} \breq \FZ_1(\CM)^{\boxtimes 3}$, which implies $(\frac{\dim(\CM)}{\dim(\CL)})^3 = p$, as global dimensions of fusion categories are totally positive. However, this would imply that $\frac{\dim(\CM)}{\dim(\CL)} = \sqrt[3]{p}$, which is not totally positive, and so we get a contradiction. Now as is mentioned above, we must have $\ord([\FZ_1(\CX)]_S) = 3$.
\end{expl}

Given a fusion category $\CX$, we denote the positive square root of $\dim(\CX)$ by $\sqrt{\dim(\CX)}$. Note that $\sqrt{\dim(\CX)}$ is totally real, but it may not be totally positive. For example, if $\CX = \rep(\Zb/p\Zb)$ for some prime $p$, then $\sqrt{\dim(\CX)} = \sqrt{p}$ is not totally positive. Note also that in this case, $\CX \oeq \overline{\CX}$, as $\CX$ has a braiding.

Let $\Pi_{0} := \left\langle\left\{[\FZ_{1}(\rep(\Zb/p\Zb))]_{S_{0}}\mid p \text{ is prime}\right\}\right\rangle$ 
be the subgroup of $\ker(\omega_{S_{0}})$ generated by the $S_{0}$-Witt classes of Drinfeld centers of prime-order cyclic group representation categories. We collect several properties of $\ker(\omega_{S_{0}})$ as follows.

\begin{prop}\label{prop:S0}
\begin{itemize}
\item[(1)] If $\CA \in \nbfc$ satisfies $\CA\weq_{S_{0}}\vect$, then the positive square root of $\sqrt{\dim(\CA)}$ is totally positive.
\item[(2)] For all $\CX \in \fus$ such that $\CX \mrt \overline{\CX}$, we have $[\FZ_{1}(\CX)]^2_{S_{0}} = [\vect]_{S_{0}}$.
\item[(3)] For all $\CX \in \fus$ such that $\CX \mrt \overline{\CX}$ and $\sqrt{\dim(\CX)}$ is not totally positive, we have 
$$\left\langle[\FZ_{1}(\CX)]_{S_{0}}\right\rangle \simeq \Zb/2\Zb \subset \ker(\omega_{S_{0}})\,.$$
\item[(4)] $\Pi_{0} \simeq (\Zb/2\Zb)^{\oplus \Nb}\subset \ker(\omega_{S_{0}})$.
\end{itemize}
\end{prop}

\begin{proof}
(1) Suppose $\CA\weq_{S_{0}}\vect$, then there exist $\CX, \CY \in \fus$ such that 
$$\CA \bt \FZ_1(\FZ_1(\CX)) \cong \FZ_1(\FZ_1(\CY))\,,$$
which implies $\dim(\CA) = \left(\frac{\dim(\CY)}{\dim(\CX)}\right)^4$. By \cite[Rmk.~2.5]{ENO05}, we firstly have $\sqrt{\dim(\CA)}=\left(\frac{\dim(\CY)}{\dim(\CX)}\right)^2$ is totally positive, and then we have the positive square root of $\sqrt{\dim(\CA)}$ is equal to $\frac{\dim(\CY)}{\dim(\CX)}$, which is again totally positive.

(2) The assumption $\CX \mrt \overline{\CX}$ implies $\FZ_{1}(\CX) \breq \FZ_{1}(\overline{\CX}) \breq \FZ_{1}(\CX)^{\rev}$. So
$$
\FZ_{1}(\CX) \bt \FZ_{1}(\CX) \breq \FZ_{1}(\CX) \bt \FZ_{1}(\overline{\CX}) \breq \FZ_{1}(\CX)\bt\FZ_{1}(\CX)^\rev \breq \FZ_{1}(\FZ_{1}(\CX)) \weq_{S_{0}} \vect\,,
$$    
where the $S$-equivalence on the right hand side follows from Lemma \ref{lem:trival-center}. 

(3) The statement follows from (1), (2) and $\dim(\FZ_1(\CX)) = \dim(\CX)^2$.

(4) As is discussed above, for any prime number $p$, $\rep(\Zb/p\Zb)$ satisfies the conditions of (3), so $\left\langle[\FZ_{1}(\rep(\Zb/p\Zb))]_{S_{0}}\right\rangle \simeq \Zb/2\Zb \simeq \mathbb{F}_{2}$, the finite field of 2 elements. Since for any collection of distinct primes $p_{1}, ..., p_{n}$, $\sqrt{\prod_{j=1}^{n} p_{j}}$ is not totally positive, so by (3) and the fact that for any distinct primes $p$ and $q$, $\rep(\Zb/p\Zb) \boxtimes \rep(\Zb/q\Zb) \breq \rep(\Zb/pq\Zb)$, we can conclude that $\left\{[\FZ_{1}(\rep(\Zb/p\Zb))]_{S_{0}}\mid p \text{ is prime}\right\}$ is a linearly independent set over $\mathbb{F}_{2}$. Now we are done by the infinitude of prime numbers.
\end{proof}

We given examples of fusion categories which are not Morita equivalent to their opposites.

\begin{expl}\label{ex:TY}
Let $A$ be a finite abelian group (written multiplicatively). A Tambara-Yamagami category \cite{Tambara1998} is a fusion category whose group of invertible objects is $A$, and whose set of non-invertible simple object is a singleton $\{m\}$, that satisfies the following fusion rules: $a \otimes b \cong ab$, $a \otimes m \cong m \otimes a \cong m$, $m \otimes m \cong \bigoplus_{a \in A} a$ for all $a, b \in A$. It is shown in op.\ cit.\  that such a category is determined by a nondegenerate symmetric bicharacter $\chi$ on $A$ and a choice of square root $\tau$ of $1/|A|$, and is denoted by $\EuScript{TY}(A, \chi, \tau)$. For example, an Ising fusion category is a Tambara-Yamagami category with $A = \Zb/2\Zb$. 

In \cite{GNN09}, the modular data of $\FZ_1(\EuScript{TY}(A, \chi, \tau))$ is explicitly given. By the Eilenberg-MacLane Theorem \cite{EM471}, the pointed subcategory of $\FZ_1(\EuScript{TY}(A, \chi, \tau))$ is completely determined by the quadratic form $Q_{\chi}: A \times \Zb/2\Zb \to \Cb, (a, x) \mapsto \chi(a,a)^{-1}$.

It is easy to see that for any Tambara-Yamagami category $\EuScript{TY}(A, \chi, \tau)$, the pointed subcategory of $\FZ(\EuScript{TY}(A, \chi, \tau))^{\rev}$ is determined by $Q_{\chi}^{-1}$. Therefore, as long as $Q_{\chi}$ is not equivalent to $Q_{\chi^{-1}}$ as quadratic forms,  $\FZ(\EuScript{TY}(A, \chi, \tau))$ cannot be equivalent to $\FZ(\EuScript{TY}(A, \chi, \tau))^{\rev}$, and consequently $\EuScript{TY}(A, \chi, \tau)$ cannot be Morita equivalent to $\overline{\EuScript{TY}(A, \chi, \tau)}$. For example, for any prime $p \equiv 3\mod 4$, the Tambara-Yamagami category $\EuScript{TY}(\Zb/p\Zb, \chi, \tau)$ with $\chi(a,b) = \zeta_p^{ab}$ and $\tau = \sqrt{p}$ is not Morita equivalent to $\overline{\EuScript{TY}(\Zb/p\Zb, \chi, \tau)}$ as -1 is not a quadratic residue mod $p$.
\end{expl}

%\begin{prob}Compute the commutative monoid $\ker(\omega_S)$ and $\Witt_S$ for $S=\sfc$, $\bfc$, $\nbfc$ and $S_0$. Do they have finite order elements of odd order?\end{prob}

When $S = \bfcfus$, by definition, $\FZ_{1}(\CX)\weq_S \vect$ means that there exists $\CA$, $\CB \in \bfcfus$ such that 
\begin{equation*}
\FZ_{1}(\CX) \boxtimes \FZ_{1}(\CA) \breq \FZ_{1}(\CB)\,, \quad\quad 
\mbox{or equivalently}  \quad\quad \CX \boxtimes \CA \mrt \CB. 
\end{equation*}
In particular, if $\CX \in \fus$ is Morita equivalent to some braided fusion category, then $\FZ_{1}(\CX)\weq_S \vect$.

\begin{rem}
Let $\CA$ be a braided fusion category. Recall that the braiding yields a tensor equivalence $\CA \oeq \overline{\CA}$, so if a fusion category $\CX$ is Morita equivalent to $\CA$, then $\CX$ must be Morita equivalent to $\overline{\CX}$ via $\CX \mrt \CA \oeq \overline{\CA} \mrt \overline{\CX}$. In particular, the Tambara-Yamagami category $\EuScript{TY}(\Zb/p\Zb, \chi, \tau)$ in Example \ref{ex:TY} cannot be Morita equivalent to any braided fusion category.
\end{rem}

\begin{prop}
When $S=\nbfcfus$, the group $\ker(\omega_S)$ is an infinite abelian group. 
\end{prop}
\begin{proof}
  Let $G$ be a non-abelian finite simple group. By Proposition \ref{prop:S0}, $[\FZ_1(\rep(G))]^2_{S_0} = [\vect]_{S_0}$, which, by Lemma \ref{lem:contain}, implies that $[\FZ_1(\rep(G))]^2_{S} = [\vect]_{S} \in \ker(\omega_S)$. We argue that in this case, $[\FZ_1(\rep(G))]_S  \ne [\vect]_S$. If not, then there exist $\CA$, $\CB \in S$ such that 
\begin{equation}\label{eq:simple}
\FZ_1(\rep(G)) \boxtimes \CA \boxtimes \CA^{\rev} \breq \FZ_1(\rep(G)) \boxtimes \FZ_1(\CA) \breq \FZ_1(\CB) \breq \CB \boxtimes \CB^{\rev}\,.
\end{equation}
According to \cite{GN08}, if $G$ is simple, then the universal grading group of $\rep(G)$ is trivial. Then by \cite[Thm.~2.2]{Nik19}, the fact that $\FZ_1(\CB) \breq \CB \boxtimes \CB^{\rev}$ contains $\rep(G)$ as a subcategory implies that $\CB$ contains $\rep(G)$ as a subcategory. This in turn implies that $\CB^{\rev}$ contains also $\rep(G)$ as $\rep(G)$ is symmetric. So the right hand side of \eqref{eq:simple} contains $\rep(G)\bt\rep(G)$. By dimension counting, $\FZ_1(\rep(G))$ does not contain $\rep(G) \bt \rep(G)$, so we must have $\CA\bt\CA^\rev$ contains $\rep(G)$. By the same argument as above, $\FZ_1(\CA)\breq\CA\bt\CA^\rev$ also contains $\rep(G) \boxtimes \rep(G) = \rep(G\times G)$. Therefore, for any $\CA$, $\CB \in \nbfc$ satisfying Eq.\ \eqref{eq:simple}, both sides of the equation contain $\rep(G)^{\bt 3}$, and in particular, the global dimensions of the categories on both sides the equation are at least $|G|^3$. 

Now if we condense $\1 \bt \fun(G\times G)$ on the left hand side of Eq.\ \eqref{eq:simple} and condense $\fun(G\times G)$ on the right hand side, we will obtain another pair of nondegenerate braided fusion categories of smaller global dimensions satisfying Eq.\ \eqref{eq:simple}. By the above discussions, we can keep condensing until we end up with two nondegenerate braided fusion categories satisfying \eqref{eq:simple} but the global dimensions of the categories on both sides of the equation are smaller than $|G|^3$, and we have reached a contradiction.

Finally, since $\FZ_1(\rep(G))$ has no proper nontrivial fusion subcategories other than $\rep(G)$ \cite[Theorem~1.2,1.3]{NNW09}, so using similar argument as above, we can conclude that if $G$ and $G'$ are distinct finite simple non-abelian groups, then $[\FZ_1(\rep(G))]_S \ne [\FZ_1(\rep(G'))]_S$. Now we are done by the fact that there are infinitely many non-abelian finite simple groups.
\end{proof}

\section{Summary and outlook}
In this paper, aiming at a more controlled way to approach the classification problem of (braided) fusion categories and potential applications in the study of topological orders and topological phase transitions \cite{KZ18,KZ20,KZ21-gapless-2, CJKYZ20,CW23,LY23,CJW22}, we provide an idea of generalizing the Morita equivalence and the Witt equivalence. The equivalence relations that are less refined than the usual Morita equivalence among fusion categories produce abelian groups if we quotient these equivalence relations from the commutative monoid of the equivalence classes of fusion categories. They also lead to equivalence relations among nondegenerate braided fusion categories that are more refined than the usual Witt equivalence, and reveal some hidden structures within each Witt class. 

The idea of generalizing the Morita equivalence and the Witt equivalence has some natural physical motivations and interpretations \cite[Section\ VIII.D \& VIII.E]{KW14}, which work in more general contexts. In this section, we briefly outline some possible generalizations in a couple of directions, to which this idea naturally applies. 

\subsection{Generalizations to (braided) fusion categories over \texorpdfstring{$\CE$}{}}

Let $\CE$ be a symmetric fusion category. A fusion category over $\CE$ is a pair $(\CX, T)$ where $\CX$ is a fusion category and $T: \CE \to \FZ_1(\CX)$ is a braided monoidal functor such that the composition of $T$ and the forgetful functor $F: \FZ_1(\CX) \to \CX$ is fully faithful \cite{DNO}. Let $\fus_\CE$ be the underlying 
set of equivalence classes of fusion category over $\CE$. It can be endowed with the structure of a commutative monoid via $\boxtimes_\CE$, the tensor product over $\CE$. A braided fusion category over $\CE$ is a braided fusion category $\CA$,  together with a braided tensor embedding $T: \CE \to \FZ_2(\CA)$, and $\CA$ is called nondegenerate over $\CE$ if $T$ is an equivalence. We denote by $\NBFC_{/\CE}$ the 2-category of nondegenerate braided fusion categories over $\CE$. It is symmetric monoidal with the tensor unit defined by $\CE$ and the tensor product defined by $\boxtimes_\CE$. We denote its underlying commutative monoid by $\nbfc_{/\CE}$.

If $\CX$ is a fusion category over $\CE$, then $\FZ_1(\CX)$ contains $\CE$ as a fusion subcategory. The M\"uger centralizer of $\CE$, denoted by $\FZ_1(\CX, \CE)$, is called the relative center of $\CX$. It is easy to see that $\FZ_1(\CX, \CE)$ is nondegenerate over $\CE$. 

A notion of Witt${}_{/\CE}$ equivalence relation between two nondegenerate braided fusion category over $\CE$ was introduced in \cite{DNO}. 
We give a generalization of this notion below.
\begin{defn}
Let $S \subset \fus_\CE$ be a nonempty subset closed under $\boxtimes_\CE$.  Two nondegenerate braided fusion categories $\CA$, $\CB$ over $\CE$ are \emph{$S$-Witt${}_{/\CE}$ equivalent} if there exist fusion categories $\CX$, $\CY$ over $\CE$ such that 
\begin{equation}
  \label{eq:2}
  \CA \boxtimes_\CE \FZ_1(\CX, \CE) \breq \CB \boxtimes_\CE \FZ_1(\CY, \CE)\,.
\end{equation}
In this case, we write $\CA\weq^{\CE}_{S} \CB$. Two fusion categories $\CL$, $\CM$ over $\CE$ are called \emph{S-Morita${}_{/\CE}$ equivalent} if $\FZ_1(\CL, \CE) \weq^\CE_S \FZ_1(\CM, \CE)$. The $S$-Witt${}_{/\CE}$ equivalence class of $\CA$ over $\CE$ is denoted by $[\CA]_{S}^{\CE}$.
\end{defn}

Clearly, $\weq^\CE_S$ is an equivalence relation on $\nbfc_{/\CE}$, and when $S = \fus_\CE$, we recover the Witt equivalence over $\CE$ defined in \cite{DNO}. Similar as before, the set of $S$-Witt${}_{/\CE}$ equivalence classes over $\CE$, denoted by $\Witt_S(\CE)$, has the structure of a commutative monoid. It was shown in \cite{DNO} that $\Witt(\CE) := \Witt_{\fus_{\CE}}(\CE)$ is an abelian group, which is called the Witt group over $\CE$. We have a surjective monoid homomorphism $\Witt_{S}(\CE) \to \Witt(\CE)$ by sending $[\CA]_S^{\CE}$ to $[\CA]^{\CE} := [\CA]_{\fus_{\CE}}^{\CE}$.

\begin{prob}
  Find interesting submonoids of $\fus_{\CE}$ that makes $\Witt_{S}(\CE)$ a group, and determine its group structure.
\end{prob}

\subsection{Generalizations to (braided) fusion higher categories}
According to the mathematical theory of topological orders in higher dimensions \cite{KW14,KWZ15,JF22,KZ22}, a nondegenerate braided fusion $n$-category describes an anomaly-free $n+$2D topological order; and a fusion $n$-category describes a potentially anomalous $n+$1D topological order. The idea of generalizing the Morita equivalence and the Witt equivalence naturally applies to these cases. We briefly outline the idea and point out some new phenomena in higher dimensional cases. 

\begin{rem}
The classification of nondegenerate braided fusion $n$-categories is equivalent to that of invertible objects in $(n+3)\vect$ \cite{JF22,KZ21-QL2}. 
\end{rem}

Let $\fus_n$ be the commutative monoid of the equivalence classes of fusion $n$-categories with the identity given by $n\vect:=\Sigma^n\Cb$ \cite{GJF19} and the multiplication given by a generalization of Deligne's tensor product that is well defined for $\Cb$-linear Karoubi complete $n$-categories \cite{GJF19,JF22}, still denoted by $\boxtimes$. The notion of Morita equivalence `$\mrt$' between two (multi-)fusion $n$-categories is well defined \cite{JF22,KZ22}. 

Let $S$ be subset of $\fus_n$ that is closed under the multiplication $\boxtimes$. For $n>1$, similar to the examples of submonoids in $\fus$ discussed in Section\,\ref{sec:sub-monoids}, there are more natural choices of submonoids of $\fus_n$ consisting of those monoidally equivalent to $E_m$-fusion $n$-categories for $m\geq 1$ \cite{JF22,KZ22}. We denote these submonoids of $\fus_n$ by ${}^{E_m}\fus_n$. %\red{How about $\fus_n(E_m)$?}
\begin{defn}
Two nondegenerate braided fusion $n$-categories $\CA$ and $\CB$ are called \emph{$S$-Witt equivalent}, denoted by $\CA \weq_S \CB$, if there exist two fusion $n$-categories $\CP, \CQ\in S$ such that
\be 
\CA \boxtimes \FZ_1(\CP) \breq \CB \boxtimes \FZ_1(\CQ)\,.
\ee
\end{defn}

\begin{rem}
When $S={}^{E_m}\fus_n$, the above definition of $S$-Witt equivalence relation is the mathematical reformulation of \cite[Def.\ 24]{KW14}. 
\end{rem}

\begin{defn}
Two fusion $n$-categories $\CL$ and $\CM$ are called \emph{$S$-Morita equivalent}, denoted by $\CL \mrt_S \CM$ if there exists $\CP,\CQ \in S \subset \fus_n$ such that 
$$
\CL \boxtimes \CP  \mrt  \CM \boxtimes \CQ\,.
$$ 
\end{defn}
When $S$ consists of only the identity of $\fus_n$, $\mrt_S$ is precisely the usual Morita equivalence $\mrt$. 

\begin{rem}
When discussing the relation between the $S$-Morita equivalence and the $S$-Witt equivalence, one should be more careful when $n>1$. Although it is still true that two Morita equivalent fusion $n$-categories share the same $E_1$-center, the converse statement is not true \cite{KZ21-QL2} unless we generalize the notion of Morita equivalence as in \cite{KZ21-gapless-2}. This generalization of Morita equivalence amounts to define the Morita equivalence for the higher categories of enriched fusion higher categories \cite{KZ21-gapless-2,KZ22}. The physical reason behind it is the fact that two gapped boundaries of the same bulk topological order cannot be connected by a gapped domain wall in general. In other words, a domain wall connecting such two gapped boundaries is necessarily gapless. 
\end{rem}

\bibliographystyle{alpha}
\bibliography{MyRef}

\end{document}